\documentclass{amsart}
\usepackage{amssymb,latexsym, amsmath, amscd, array, graphicx}

\usepackage{graphicx}
\usepackage[all]{xy}
\usepackage{amscd}

\swapnumbers
\numberwithin{equation}{section}

\theoremstyle{plain}

\newtheorem{thm}{Theorem}[subsection]
\newtheorem{cor}[thm]{Corollary}

\newtheorem{question}[thm]{Problem}
\newtheorem{prop}[thm]{Proposition}

\theoremstyle{definition}
\newtheorem{defin}[thm]{Definition}

\newtheorem{remark[thm]}{Remark}

\def\TC{{\mathop\mathrm{TC}\,}}
\def\dTC{{\mathop\mathrm{dTC}\,}}

\def\cat{\protect\operatorname{cat}}

\def\cd{\protect\operatorname{cd}}

\def\Z{{\mathbb Z}}

\def\1{\hbox{\rm\rlap {1}\hskip.03in{\rom I}}}
\def\Bbbone{{\rm1\mathchoice{\kern-0.25em}{\kern-0.25em}
{\kern-0.2em}{\kern-0.2em}I}}

\long\def\forget#1\forgotten{} %

\begin{document}

\title[On distributional topological complexity of groups and manifolds]
{On distributional topological complexity of groups and manifolds}
\author[A.~Dranishnikov]
{Alexander Dranishnikov}
\address{{Alexander Dranishnikov, Department of Mathematics, University
of Florida, 358 Little Hall, Gainesville, FL 32611-8105, USA}}
\email{dranish@math.ufl.edu}

\thanks{}

\subjclass[2000]
{Primary  55M30,  
Secondary 57N65, 20K01, 20F67  
}

\begin{abstract} 
We prove the equality $\dTC(\Gamma)=\TC(\Gamma)$   for distributional topological complexity of torsion free hyperbolic and of torsion free nilpotent groups.

For the distributional topological complexity of lens spaces we prove the inequality  $\dTC(L^n_p)\le 2p-1$ and for the distributional LS-category the inequality $d\cat(L^n_p)\le p-1$ where the latter turns into equality for prime $p$ and $n>p$.  We use these inequalities to construct counter-examples to the product formula for $d\cat$ and $\dTC$.
\end{abstract}

\maketitle

\section{Introduction}
The topological complexity TC$(X)$ of a configuration space $X$ is a numerical invariant that estimates the difficulty  of constructing a navigation algorithm in $X$. A navigation algorithm on $X$ is a map $f:X\times X\to P(X)$ from the product $X\times X$ to the space $P(X)=Map([0,1]\to X)$ of all paths in $X$ such that for each pair $(x_0,x_1)\in X\times X$
the path $f(x_0,x_1):[0,1]\to X$ satisfies $f(x_0,x_1)(0)=x_0$, the initial state of a robot, and the path $f(x_0,x_1)(1)=x_1$, the terminal state. Clearly, if $X$ admits a continuous navigation algorithm, then $X$  as a topological space is contractible. It is a very restrictive condition. 

The topological complexity  TC$(X)$ was defined by Farber as the minimal number of pieces (minus one nowadays) needed to cover $X\times X$ in such a way that on each piece there is a continuous navigation algorithm~\cite{F1,F2}.
The numerical invariant $\TC$ is similar in spirit to an almost a hundred-year-old invariant called the Lusternik-Schnirelmann category, $\cat X$~\cite{LS,CLOT}. 
It was introduced by Lusternik and Schnirelmann in their solution of  Poincar\'e's problem concerning the number of minimal geodesics on  a 2-sphere~\cite{LS}. Since then the LS-category has found many applications in different areas of mathematics. 

Both $\TC$ and $\cat$ are homotopy invariant: if a topological space $X$ is homotopy equivalent to $Y$ then
TC$(X)$=TC$(Y)$ and $\cat X=\cat Y$.  It means that both invariants can be defined for discrete groups $G$ as TC$(G)=$TC$(BG)$ and $\cat(G)=\cat(BG)$ where $BG$ is a classifying CW complex for $G$ or equivalently $BG$  a $K(G,1)$-space. Since for fixed $G$ all $K(G,1)$-complexes are homotopy equivalent, the group invariants
TC$(G)$ and $\cat(G)$ are well-defined. In the '50s Eilenberg and Ganea proved that $\cat(G)$ equals the cohomological dimension $\cd(G)$ of $G$~\cite{EG}.
Thus, $\cat(G)$ does not provide a new invariant to group theory. On the other hand TC$(G)$ is quite intriguing invariant which is not yet been computed for many classes of groups.

Recently, Knudsen and Weinberger~\cite{KW1} and independently Jauhari and the author~\cite{DJ} have modified $\cat$ and $\TC$ to their probabilistic versions obtaining new invariants:
 the distributional topological complexity, $\dTC$ and the distributional LS-category $d\cat$. Knudsen and Weinberger used the terms {\em analog category} and {\em analog} complexity where 
the word ``analog"  emphasizes the opposite of ``digital". These new concepts already were studied, in addition to~\cite{DJ} and~\cite{KW1} in the papers~\cite{Dr2, KW2, J1, J2, JO1, JO2, DaJ}.

Informally the difference between the classical $\cat$ and the new $d\cat$ can be explained as follows. 
By definition, $\cat X\le n$ if there exists an open cover $\mathcal U=\{U_0,\dots,U_n\}$ of $X$ together with contractions of each $U_i$ to a base point $x_0$. A contraction of $U_i$ to $x_0$ can be seen as a family of paths $\phi_x^i$ from $x\in U_i$ to $x_0$ that depends continuously on $x$. Using  a partition of unity subordinated by $\mathcal U$ one can assign weights $\lambda_x^i\in[0,1]$ to each path $\phi_x^i$, $0\le i\le n$, $x\in X$, depending continuously on $x$ and satisfying $\sum_i\lambda_x^i=1$.
This yields an ordered multi-path $$\Phi_x=(\lambda_x^0\phi_x^0,\lambda_x^1\phi_x^1,\dots,\lambda_x^n\phi_x^n)$$ of cardinality $n+1$ from $x\in X$ to $x_0$ depending continuously on $x$. Thus, the collection of all multi-paths $\Phi_x$, $x\in X$, defines a continuous multi-path contraction of $X$ to $x_0$. It is not difficult to see that the inequality $\cat X\le n$ is equivalent to the existence of a continuous contraction of $X$ to $x_0$ by means of ordered multi-paths of cardinality $n+1$~\cite{F3}. In the definition of the distributional category $d\cat$ we simply drop the requirement of ordering. Clearly, it yields the inequality $d\cat\le\cat$. We note that by loosing the order turns a multi-path into a probability measure on the space of paths. 

The same idea is used in the definition of distributional topological complexity. A navigation algorithm in in this setting is a continuous choice of distributed paths between any $x$ and $y$ in the configuration space $X$. In the language of robotics, this means that
for each $x$ and $y$ we have at most $n+1$ paths with continuously distributed probabilities, allowing a robot to choose a certain path, such that the family of all ``probabilistic" multi-paths continuously depends on $(x,y)\in X\times X$.

\

It turns out that the distributional analogs can differ significantly from the classical invariants, both for groups and for spaces.
In particular, it was proved in both~\cite{DJ} and ~\cite{KW1} that $\dTC(\mathbb Z_2)=d\cat(\mathbb Z_2)=1$ whereas TC$(\mathbb Z_2)=\cat(\mathbb Z_2)=\cd(\mathbb Z_2)=\infty$. Moreover, we proved that $$\dTC(\mathbb RP^n)=d\cat(\mathbb RP^n)=1$$ for all $n$. Note that in the classical case $\cat(\mathbb RP^n)=n$
and TC$(\mathbb RP^n)$ equals, with the exception $n=1,3,7$, the immersion number for $\mathbb RP^n$~\cite{FTY}.

Knudsen and Weinberger proved~\cite{KW1} that for torsion free groups $\Gamma$ the distributional category coincides with the Lusternik-Schnirelmann category, i.e., $d\cat\Gamma=\cat\Gamma$.
It remains an open problem whether the similar equality $\dTC(\Gamma)=$TC$(\Gamma)$ holds for topological complexity. 
In the first part of this paper we prove the following. 

\

{\bf Theorem A.} (Theorem~\ref{Frob}) {\em For a Frobenius injective group} $\pi$  $$\dTC(\pi)=\TC(\pi).$$
A group $G$ is called {\em Frobenius injective} if the  Frobenius map $f:G\to G$, $f(g)=g^k$, is injective for every $k$.
We prove that  torsion free hyperbolic groups and torsion free nilpotent groups are Frobenius injective.
Therefore, for non-orientable surfaces of genus $g\ge 3$ we have the equality $\dTC(N_g)=4$.
Thus, for non-orientable surfaces the only remaining open question is about the value of $\dTC$ for the Klein bottle.
We recall that determining the classical $\TC$ in this case was quite challenging problem until it was settled in~\cite{CV}.

\

It turns out that for finite groups probabilistic versions of the LS-category and the topological complexity differ drastically from their classical counterparts.
In~\cite{KW1} Knudsen and Weinberger proved that
$$d\cat(G)\le \dTC(G)\le|G|-1$$ for any finite group $G$. Later it was shown that $d\cat(\mathbb Z_p)=\dTC(\mathbb Z_p)=p-1$~\cite{Dr2} and, moreover that $$d\cat(G)=\dTC(G)=|G|-1$$ for any $p$-group~\cite{KW2}.  

Thus, for an infinite-dimensional lens space $L^\infty_p$ we have $$d\cat(L^\infty_p)=\dTC(L^\infty_p)=p-1.$$
As mentioned earlier, for $p=2$ this equality holds for finite-dimensional lens spaces. For prime $p>2$, this remains an open problem. 

In this paper, for finite-dimensional lens spaces $L^m_r$ we obtain the following upper bounds where $r$ is non necessarily prime.

\

{\bf Theorem B.} (Corollary~\ref{upper dcat} and Theorem~\ref{TC upper}) {\em For all $m$ and  $r$, $$d\cat(L^m_r)\le r-1$$ 
and for odd $r$,}
 $$
\dTC(L^m_r)\le\begin{cases}r-1\ & if\ m\le r-1\\ 
m\ & if\ m\le 2r-1\\
2r-1\ & if\ m\ge 2r-1.\\
\end{cases}
$$
We note that our upper bound for $\dTC$ is not always optimal.

In the infinite-dimensional case the proof of the lower bounds is based on the equivariant Sullivan conjecture, which is a theorem due to Gunnar Carlsson~\cite{C} and others~\cite{DMN,La}. This powerful technique does not apply in the finite-dimensional case. In this paper, we obtain lower bounds for finite-dimensional lens spaces using classical methods known as Borsuk--Ulam-types theorems. In the cases when our lower and upper bounds agree, we obtain the following 

\

{\bf Theorem C.} (Theorems~\ref{=} and~\ref{p^k}) {\em For prime $p$ and  $m\ge p$,
 $$d\cat(L^m_p)=p-1.$$ 
For odd prime $p$ and $m\ge kp^k-(k+2)p^{k-1}+3$,
$$d\cat(L^m_{p^k})= p^k-1.$$ }

At the end of the paper we use our upper and lower bounds  for  the distributional category and topological complexity of lens spaces to produce counterexamples to the product formulas for $d\cat$ and $\dTC$ of manifolds.
The product formulas (inequalities)
$$\cat(X\times Y)\le\cat X+\cat Y\ \ \ \text{and}\ \ \ \ \TC(X\times Y)\le\TC(X)+\TC(Y)$$ are well known for the classical category and  topological complexity~\cite{CLOT,F2}.

\section{Preliminaries}
\subsection{$\TC$  and $\cat$ for spaces and groups} The definition of the LS-category $\cat(X)$ of a topological space was given in the introduction.
By the Eilenberg-Ganea theorem, for groups we have $\cat(\Gamma)=\cd{\Gamma}$.

The topological complexity $\TC(X)$ of a space $X$ is the minimal number $n$ such that $X\times X$ can be covered by $n+1$ open sets $U_0,\dots, U_n$, each admitting a continuous map $s_i:U_i\to P(X)$ to the path space on $X$ satisfying  $s_i(x,y)(0)=x$ and $s_i(x,y)(1)=y$ for all $(x,y)\in U_i$.

For a group $\Gamma$, the invariant $\TC(\Gamma)=\TC(B\Gamma)$ is difficult to compute. For groups with finite $B\Gamma$, the following characterization holds~\cite{FGLO}
\begin{thm}\label{Farber}
For a group $\Gamma$ with finite $B\Gamma$, the topological complexity $\TC(\Gamma)$ equals the minimal integer $k$ such that the canonical map
$$\mu:E(\Gamma\times\Gamma)\to E_{\mathcal D}(\Gamma\times\Gamma)$$ is $(\Gamma\times\Gamma)$-equivariantly homotopic to a map whose image lies in
the $k$-dimensional skeleton $E_{\mathcal D}(\Gamma\times\Gamma)^{(k)}$.
\end{thm}
Here  $\mathcal D$ is the family of subgroups of $\Gamma\times\Gamma$ obtained from the diagonal  subgroup $\Delta\Gamma$  by conjugation and finite intersections, and containing the trivial subgroup. The space $E_{\mathcal D}(\Gamma\times\Gamma)$ is a classifying space for $(\Gamma\times\Gamma)$-actions with isotropy groups in $\mathcal D$.
\subsection{Probability measures with finite supports}
Let $\mathcal B_n(X)$ denote the space of probability measures on $X$ supported on at most $n$ points.
Although this is often denoted $P_n(X)$, we use $\mathcal B_n(X)$ to avoid confusion with path spaces.

The standard topology on $\mathcal B_n(X)$ is the quotient topology induced from the symmetric join product~\cite{KK},\cite{V}
 $$Symm^{\ast n}(Z):=\ast^nZ/S_n$$ which is the orbit space of the action of symmetric group $S_n$ on the iterated join product $\ast^nZ=Z\ast\cdots\ast Z$. 
 Elements of $Symm^{\ast n}(Z)$ can be viewed as formal sums $t_1x_1+\cdots t_nx_n$ with no ordering. 
 The quotient map $q:Symm^{\ast n}(X)\to \mathcal B_n(Z)$  identifies terms via $tx+t'x=(t+t')x$. We note that $q$ is a map with compact contractible fibers. 

When $X$ is a metric space, there are choices of metrics on $\mathcal B_n(X)$ like Kantorovich--Rubinshtein metric, Wasserstein metric, or
Levy--Prokhorov metric which generally produce coarser topologies. In robotics $X$ is the space of states and usaually it is metrizable. Then it makes sense to use
one of the natural metrics on $\mathcal B_n(X)$ extending the metric on $X$. For general groups $B\Gamma$ is not metrizable. In that case we will assume the quotient
topology as it was introduced in~\cite{KW1}. In this paper we will continue to call our invariants  ``distributional" instead of ``analog".

Clearly $\mathcal B_n(-)$ is a covariant functor on the category of topological spaces. 
Thus, for a continuous map $f:X\to Y$ we have the push-forward map $$\mathcal B_n(f):\mathcal B_n(X)\to\mathcal B_n(Y)$$
which is continuous.
There is a natural map $i_X:X\to\mathcal B_n(X)$ defined by the Dirac measures $x\mapsto \delta_x$ which is an embedding for reasonable topological spaces $X$. 
\begin{defin}
We say that a map $f:X\to Y$ admits a $\mathcal B_n$-section if there exists a map $\mu:Y\to\mathcal B_n(X)$ such that $\mathcal B_n(f)\circ\mu=i_Y$.
\end{defin}
We denote by $$M_n(X,f)=\mathcal B_n(f)^{-1}(Y)\ \ \ \text{and by} \ \ \ M_n(f):M_n(X,f)\to Y$$ the restriction of $\mathcal B_n(f)$ to $M_n(X,f)$.
If $f$ is a fibration with a fiber $F$, then $M_n(f)$ is a fibration with fiber $\mathcal B_n(F)$~\cite{DJ,KW1}.
Thus, a fibration $f$ admits a $\mathcal B_n$-section if and only $M_n(f)$ admits a section.

\subsection {$d\cat$ and $\dTC$ for spaces}
We refer to the original definitions of distributional category and distributional topological complexity to~\cite{DJ,KW1}.
In this paper we will be using equivalent definitions formulated  below. 

Let $P(X)=Map([0,1],X)$ be the space of paths on $X$ and let $P_0(X)$ be the subspace of paths issued from a base point $x_0\in X$.
Let $p:P(X)\to X\times X$ be the end-points map and $p_0:P_0(X)\to X$ the restriction of $p$. We note that $p$ and $p_0$ are Hurewicz fibration with the fiber homotopy equivalent to the loop space $\Omega X$.

\begin{defin}\label{dcat}
$d\cat X\le n$ $\Leftrightarrow$ the fibration $p_0:P_0(X)\to X$ admits a $\mathcal B_{n+1}$-section.

$\dTC X\le n$ $\Leftrightarrow$ the fibration $p:P(X)\to X\times X$ admits a $\mathcal B_{n+1}$-section.
\end{defin}
We denote by $E_n(X)=M_{n+1}(X,p_0),\ \ \ p_n=M_{n+1}(p_0)$ and  by 

$F_n(X)=M_{n+1}(X\times X,\bar p),\ \ \ \bar p_n=M_{n+1}(p)$. 
Then Definition~\ref{dcat} 
can be reformulated as follows:

$d\cat X\le n$ $\Leftrightarrow$ the fibration $p_n:E_n(X)\to X$ admits a section.

$\dTC X\le n$ $\Leftrightarrow$ the fibration $\bar p_n:F_n(X)\to X\times X$ admits a section.

\

\subsection {Characterization of $d\cat$ and $\dTC$ for groups}

Since for a group $\Gamma$ the loop space $\Omega(B\Gamma)$ is homotopy equivalent to $\Gamma$, the above characterizations take the following form.

Let $p_\Gamma:E\Gamma\to B\Gamma$ be the universal covering. 
\begin{thm}\label{dcat for group}\cite{Dr2,KW1,KW2}
$d\cat \Gamma\le n$ $\Leftrightarrow$ the universal covering $p_\Gamma:E\Gamma\to B\Gamma$ admits a $\mathcal B_{n+1}$-section.
\end{thm}
We use notation $E_n\Gamma=M_{n+1}(E\Gamma, p_\Gamma)=\mathcal B_{n+1}(\Gamma)\times_\Gamma E\Gamma$ and $p_n^\Gamma=M_{n+1}(p_\Gamma)$ for the projection
$\mathcal B_{n+1}(\Gamma)\times_\Gamma E\Gamma\to B\Gamma$ in the Borel construction.
Then Theorem~\ref{dcat for group} can be restated as

\

{\em $d\cat \Gamma\le n$ $\Leftrightarrow$ the fibration $p_n^\Gamma:E_n\Gamma\to B\Gamma$ admits a section.}

\

We note that $M_{n+1}(p_\Gamma)$ is a locally trivial bundle with fiber $\mathcal B_{n+1}(\Gamma)=\Delta(\Gamma)^{(n)}$,
homeomorphic to the $n$-skeleton of the simplex $\Delta(\Gamma)$ spanned by $\Gamma$.

\

Let $p_D:D\Gamma\to B(\Gamma\times\Gamma)$ be a covering map corresponding to the diagonal subgroup in $D\Gamma\subset\Gamma\times\Gamma$.
Note that $D\Gamma$ is a model for $B\Gamma$.
\begin{thm}\label{dTC for group}\cite{Dr2,KW1,KW2}
$\dTC(\Gamma)\le n$ $\Leftrightarrow$ the covering map $p_D:D\Gamma\to B(\Gamma\times\Gamma)$ admits a $\mathcal B_{n+1}$-section.
\end{thm}
The fiber of $p_D$ is in natural bijection with the cosets $$(\Gamma\times\Gamma)/\Gamma=\{(x\times 1)D\Gamma\mid x\in\Gamma\}$$ which are in natural bijection with $\Gamma$.
The action of $\Gamma\times\Gamma$ on $\{(x\times 1)D\Gamma\mid x\in\Gamma\}$ is given by the formula
$$
(a\times b)(x\times 1)D\Gamma=(ax\times b)D\Gamma=(axb^{-1}\times 1)D\Gamma.
$$
Thus,  $M_{n+1}(p_D)$ is a locally trivial bundle with fiber $\Delta(\Gamma)^{(n)}$ and with an action of $\Gamma\times\Gamma$ on the vertices of $\Delta(\Gamma)$
given by  the formula $(a\times b)(x)=axb^{-1}$. This yields the following characterization of $\dTC$.
\begin{thm}\label{better dTC}\cite{KW1},\cite{KW2}
$\dTC(\Gamma)\le n$ if and only if the fibration $$\Delta(\Gamma)^{(n)}\times_{(\Gamma\times\Gamma)}E(\Gamma\times\Gamma)\to B(\Gamma\times\Gamma)$$
admits a section.
\end{thm}

\begin{cor}\label{charact}
For a torsion free group, $\dTC(\Gamma)$  equals the minimal $k$ such that there exists a $(\Gamma\times\Gamma)$-equivariant map $E(\Gamma\times\Gamma)\to\Delta(\Gamma)^{(k)}$.
\end{cor}

We use the notations $$F_n\Gamma=M_{n+1}(D\Gamma,p_D)=\Delta(\Gamma)^{(n)}\times_{(\Gamma\times\Gamma)}E(\Gamma\times\Gamma)$$ and $\bar p^\Gamma_n=M_{n+1}(p_D)$. Then Theorem~\ref{dTC for group} states that

\

{\em $\dTC(\Gamma)\le n$ $\Leftrightarrow$ the  map $\bar p_n^\Gamma:F_n\Gamma\to B(\Gamma\times\Gamma)$ admits a section.}

\subsection{$d\cat$ and $\dTC$ of maps}
We define  $d\cat$ and $d\TC$ for  maps as follows. {\em The distributional topological complexity}, $d\TC(f)$, of a map $f:X\to Y$ is the minimal number $n$ such that there exists a continuous map
$$s:X\times X\to\mathcal B_{n+1}(P(Y))$$ satisfying $s(x,x')\in\mathcal B_{n+1}(P(f(x),f(x')))$ for all $(x,x')\in X\times X$.
Here $$P(y,y')=\{f\in P(Y)\mid f(0)=y,\ f(1)=y'\},$$

{\em The distributional LS-category}, $d\cat(f)$,  is the minimal number $n$ such that there is a continuous map $$s:X\to \mathcal B_{n+1}(P(Y))$$ satisfying $s(x)\in \mathcal B_{n+1}(P(f(x),y_0))$ for all $x \in X$. We note that $$d\cat(1_X)=d\cat(X)\ \ \text{and}\ \ d\TC(1_X)=d\TC(X).$$ Clearly, for $f:X\to Y$,
\begin{equation}\label{dcat of maps}
d\cat(f)\le\min\{d\cat X, d\cat Y\}\  \text{and} \ d\TC(f)\le\min\{d\TC(X),d\TC(Y)\}.
\end{equation}

The proof of Theorem 2.6 and Theorem 2.7 in~\cite{DJ} can be extended to the following.
\begin{prop}\label{liftD}
Let $f:X\to Y$ be a map. Then

(a)   $d\cat(f)\le n$ if and only if $f$ admits a lift with respect to $p_n:E_nY\to Y$;

(b) $d\TC(f)\le n$ if and only if $f\times f$  admits a lift with respect to $\bar p_n:F_nY\to Y\times Y$.
\end{prop}
Since the existence of a lift of $f$ in Proposition~\ref{liftD}(a) is equivalent to existence of a $\mathcal B_{n+1}$-section of the pull-back of $p_0^Y$ with respect to $f$,
in the case when $Y=B\Gamma$ and $p^Y$ is the universal covering we obtain the following
\begin{cor}\label{Lift}
Let $p':u^*E\Gamma\to Y$ be the pull-back of the universal covering $p_\Gamma:E\Gamma\to B\Gamma$ with respect to a map $u:Y\to B\Gamma$.
Then $d\cat(u)\le n$ if and only if $p'$ admits a $\mathcal B_{n+1}$-section.
\end{cor}

\subsection{Notations and conventions}

 Hyperbolic groups in this paper are Gromov's hyperbolic~\cite{Gr}.

We consider actions of a discrete group $G$ on CW-complex $X$ that preserve the CW-complex structure in such a way that if an open cell $e$ is taken  to itself by an element $g\in G$,
then the restrictyion of $g$  to $e$ is the identity map. Such a complex is called a $G$-CW-complex.

\section{$\dTC$ of Frobenius injective groups}

\subsection{Frobenius Injective Groups}
For a group $G$  and  $k\in \mathbb N$ a power map $f:G\to G$, $f(x)=x^k$, is called a {\em Frobenius map}.
We call a group $G$ {\em Frobenius injective} if every  Frobenius map $f:G\to G$ on $G$ is injective.
Thus, Frobenius injective groups are torsion free.

Clearly, a subgroup of Frobenius injective group is Frobenius injective. In particular,
any subgroup of a uniquely divisible group is Frobenius injective.
Examples of uniquely divisible groups include the groups of positive units
of real closed fields, unipotent matrix groups, noncommutative power series with unit constant
term, and the group of characters of a connected Hopf algebra over a field of characteristic
zero~\cite{Ba,ABS,HLR}.

By $Z(a)=\{x\in G\mid xa=ax\}$ we denote the centralizer of an element $a\in G$ in a group $G$.
For a set $S\subset G$ we denote by $Z(S)=\bigcap_{a\in S} Z(a)$.
The following was proven in~\cite{FGLO}, Lemma 8.0.4:
\begin{prop}\label{central}
For torsion free groups $G$ with cyclic centralizers either $Z(a)=Z(b)$ or $Z(a)\cap Z(b)=1$ for all $a,b\in G$.
\end{prop}
It was shown in~\cite{FGLO} that such groups $G$ satisfy the following property:

(N) For any $x\in G$ and for any finite set $S\subset G$ the condition $x^n\in Z(S)$ implies that $x\in Z(S)$.

\begin{prop} 
Every torsion free group satisfying property (N)  is Frobenius injective.
\end{prop}
\begin{proof}
Suppose that $x^k=y^k$. Then $x^k\in Z(y)$. By Property (N), $x\in Z(y)$. Hence, $(xy^{-1})^k=x^ky^{-k}=1$. Since the group is torsion free, we obtain $xy^{-1}=1$ and
thus $x=y$.
\end{proof}
\begin{cor}
Torsion free nilpotent groups are Frobenius injective.
\end{cor}
It was proven in~\cite{FGLO}, Lemma 8.0.3, that such groups have property (N).

Since torsion free hyperbolic groups have cyclic centralizers (see~\cite{BH} Corollary 3.10 in Chapter III$\Gamma$), we obtain the following:
\begin{cor}
Torsion free hyperbolic groups are Frobenius injective.
\end{cor}

Let $\pi$ be a torsion free group and let $G=\pi\times\pi$. 
The action of $G$ on $\pi$, defined by the formula $(x,y)(\gamma)=x\gamma y^{-1}$, extends to an action on the simplex $\Delta(\pi)$ spanned by $\pi$.
Then the map
$f:\Delta(G)\to\Delta(\pi)$ defined as $$f(\sum t_i(x_i,y_i))=\sum t_i(x_iy_i^{-1})$$ is $G$-equivariant:
$$
f((a,b)\sum t_i(x_i,y_i))=f(\sum t_i(ax_i,by_i))=\sum t_i(ax_iy_i^{-1}b^{-1})=(a,b)f(\sum t_i(x_i,y_i)).$$
We denote by $\mathcal D$ the family of subgroups of $G$ obtained from the diagonal subgroup $\delta\pi\subset\pi\times\pi$ and the trivial subgroup by conjugation and taking finite intersections.   It was noticed in~\cite{FGLO} that 
$
\mathcal D=\{H_{b,S}\mid b\in\pi,\ S\subset\pi,\ |S|<\infty\}
$
where $H_{b,S}=\{(a,bab^{-1})\in\pi\times\pi\mid a\in Z(S)\}$. We note that here the conjugates of the diagonal are represented by $S=\{1\}$.

\subsection{Classifying space $E_{\mathcal D}(G)$}
For a discrete group $G$ and a family $\mathcal F$ of its subgroups that is closed under conjugation and finite intersections, a classifying complex $E_{\mathcal F}G$
is a $G$-CW-complex whose isotropy subgroups belong to $\mathcal F$, and for each $G$-CW complex with isotropy groups from $\mathcal F$, there exists a $G$-equivariant map $f:X\to E_{\mathcal F}G$, unique  up to a 
$G$-homotopy. 
 The corresponding map $$f:EG\to E_{\mathcal F}G$$
for the free $G$-action on $EG$ is 
called the {\em canonical map}.

Wolfgang L\"{u}ck proved the following~\cite{Lu}.
\begin{thm}\label{Lu}
A $G$-CW-complex $X$ is $E_\mathcal FG$ if and only if for each $x\in X$, the isotropy group  of $x$ belongs to the family $\mathcal F$, and for each group $F\in\mathcal F$ the fixed point set $X^F$ is contractible.
\end{thm}
We use L\"{u}ck's theorem to prove the following
\begin{thm}\label{Frob}
For every Frobenius injective group $\pi$ the simplex $\Delta(\pi)$ spanned by $\pi$ is a classifying G-CW-complex $E_{\mathcal D}(G)$.
\end{thm}
\begin{proof}
The isotropy group of a vertex $\gamma\in\Delta(\pi)$ equals
$$H_\gamma=\{(x,y)\in \pi\times\pi\mid x\gamma y^{-1}=1\}=\{(e,\gamma)(x,x)(e,\gamma)^{-1}\}=(e,\gamma)\delta\pi(e,\gamma)^{-1},$$
which belongs to the family $\mathcal D$ as a conjugate to the diagonal subgroup $\delta\pi$.

Claim: {\em Suppose that some $(x,y)\in G$  fixes a point $$z=\sum_{i=1}^k t_i\gamma_i\in\Delta(\pi)\ \ \text{where}\ \ t_i>0.$$
Then $(x,y)$ fixes each $\gamma_i$.}

{\em Proof of the claim.} Suppose, to the contrary, that $(x,y)$ does not fix all $\gamma_i$
Then $(x,y)$ defines a permutation of $\gamma_1,\dots\gamma_k$.
We choose a nontrivial cycle of the permutation.
Without loss of generality, assume $(x,y)(\gamma_1)=\gamma_2,\ 
\dots,\ (x,y)(\gamma_{\ell-1})=\gamma_\ell,\ (x,y)(\gamma_\ell)=\gamma_1.$
Thus, $x\gamma_1y^{-1}=\gamma_2,\ \dots,\ x\gamma_{\ell-1}y^{-1}=\gamma_\ell,\ x\gamma_\ell y^{-1}=\gamma_1$.
Substituting successively, we obtain
$x^\ell\gamma_1y^{-\ell}=\gamma_1$, which can be  rewritten as $y^\ell=\gamma_1^{-1}x^\ell\gamma_1$.
Hence $(\gamma_1^{-1}x\gamma_1)^\ell=y^\ell$. By Frobenius injectivity, it follows that $\gamma_1^{-1}x\gamma_1=y$, and therefore
$(x,y)(\gamma_1)=x\gamma_1y^{-1}=\gamma_1$, a contradiction.

Repeating this argument starting from the $i$th equation yields  $y^\ell=\gamma_i^{-1}x^\ell\gamma_i$, and hence 
$(x,y)(\gamma_i)=x\gamma_iy^{-1}=\gamma_i$ for all $i\le\ell$.
Applying the same arguments to all cycles shows that $(x,y)$ fixes all vertices $\gamma_1,\dots,\gamma_k$.

Thus the isotropy group $H_z$ of $z$  is the intersection of the isotropy subgroups  of vertices $\gamma_i\in\Delta(\pi)$, $H_z=\bigcap_{i=1}^k H_{\gamma_i}$. Hence, $H_z\in\mathcal D$ for all $z$.

Now, given $H\in\mathcal D$, we show that the fixed point set $\Delta(\pi)^H$ is contractible.  As mentioned above, $H=H_{b,S}=(1,b)\delta Z(S)(1,b)^{-1}$ for some $b\in\pi$ and  a finite $S\subset \pi$, where $\delta Z(S)\subset\delta\pi\subset \pi\times\pi$ is the diagonal subgroup.
Since the fixed point set of a conjugate of $H$ is homeomorphic to the fixed point set of the group $ \delta Z(S)$, it suffices to prove contractibility for the latter.
We note that the group  $ \delta Z(S)$ fixes $\gamma\in\pi$ if and only if $\gamma$ commutes with all elements of $Z(S)$: $$(x,x)(\gamma)=\gamma\ \Leftrightarrow\ \ x\gamma x^{-1}=\gamma\ \ \text{for}\ \ x\in Z(S).$$ In view of the Claim the fixed point set of $\delta Z(S)$ is the simplex $\Delta(Z(Z(S)))$ which is contractible.

We apply L\"{u}ck's characterization theorem for classifying spaces relative to a family classifying spaces~\ref{Lu} to complete the proof.
\end{proof}

\begin{thm}\label{Frob}
For a Frobenius injective group $\pi$,  $$\dTC(\pi)=\TC(\pi).$$
\end{thm}
\begin{proof}
We show that $\TC(\pi)\le\dTC(\pi)$.
If $\dTC(\pi)\le n$ then by Corollary~\ref{charact} there is  a $G$-equivariant map $EG\to\Delta(\pi)^{(n)}$.
By Theorem~\ref{Frob} $\Delta(\pi)$ is a $E_{\mathcal D}G$ classifying space. Then by Theorem~\ref{Farber}  it follows that $\TC(\pi)\le n$.
\end{proof}
In view of the main result of~\cite{Dr1} we obtain the following
\begin{cor}
For torsion free hyperbolic groups $$\dTC(\pi)=\TC(\pi)=2\cd(\pi).$$
\end{cor}
Since the fundamental group $\pi=\pi_1(N_g)$ is torsion free for $g>1$ and hyperbolic for $g>2$  we obtain
\begin{cor}
For non-orientable surfaces of genus $g\ge 3$,$$\dTC(N_g)=4.$$
\end{cor}
A more technical proof of this fact for $g>3$ was given in~\cite{Dr2}. 
The case of the Klein bottle, $g=2$, remains open.
We note that the fundamental group $G=\pi_1(N_2)$ of the Klein bottle is not Frobenius injective,
since it has a presentation $$G=\langle a,b\mid a^2=b^2\rangle$$ which shows that
the Frobenius map  $f:G\to G$, $f(x)=x^2$ is not injective.

\begin{question}
Let $\pi$ be the fundamental group of the Klein bottle. Is $\Delta(\pi)$ a classifying space $E_{\mathcal D}(\pi\times\pi)$?
\end{question}

\section{Upper bounds for $d\cat$ and $\dTC$ of lens spaces}

Knudsen and Weinberger~\cite{KW1} proved that $\dTC(G)\le |G|-1$ for every finite group $G$. In~\cite{DJ} and~\cite{KW1} it was shown that $\dTC(\mathbb Z_2)=1$.
Moreover, in ~\cite{DJ} it was proven that $$d\cat(\mathbb RP^n)=\dTC(\mathbb R P^n)=1\ \ \text{for all}\ \ n.$$ We derived this equality from the inequality $\dTC(\mathbb RP^n)\le 1$,
the proof of which is based on an elementary geometric idea. Here we apply similar idea to lens spaces $L^m_p$ with $r>2$. For us a lens space $L^m_r$ is the orbit space of a free $\mathbb C$-linear action of the cyclic group $\mathbb Z_r$ on the unit $m$-sphere $S^m\subset\mathbb C^n$, where $m=2n-1$. 

\begin{prop}\label{2p-1}
When $r$ is odd,
$\dTC(L^m_r)\le 2r-1$ for all $m$ and  $r$ and

$\dTC(L^m_r)\le r-1$ for even $r$.
\end{prop}
\begin{proof}
Let $r$ be odd.
We denote $G=\mathbb Z_r$.
By definition, the space $L^m_r$ is the orbit space of a $G$-action on $S^m$ by isometries, where $m=2n-1$. 
Let $Gx$ and $Gy$ be two orbits. 
Here we will identify $z\in S^m$ with the Dirac measure $\delta_z$ supported at $z$.
Given $x$ and $y$ in $S^m$, we define a $\mathcal B_{2r}$-path in $S^m$ from $x$ to the uniformly distributed measure $\sum_{g\in G}\frac{1}{r}gy\in\mathcal B_r(Gy)$
on the orbit $Gy$
 by the formula
$$
\phi(t)=\sum_{g\in G}\left(\frac{\beta(x,gy)}{2\pi}R^{\alpha
(x,gy)}(t)+\frac{\alpha(x,gy)}{2\pi}R^{\beta(x,gy)}(t)\right),
$$
where $\alpha(x,y)$ and $\beta(x,y)$ are the angles between vectors $x$ and $y$,
$\alpha(x,y)+\beta(x,y)=2\pi$, and where $R^{\alpha(x,y)}(t)$ and $R^{\beta(x,y)}(t)$ are paths defined by the rotations in the plane $A$ containing  $x$ and $y$, taking $x$ to $y$. These are rotations by angles $\alpha(x,y),\beta(x,y)$ in opposite directions. We may assume that $\alpha\le\beta$.  When $\alpha=\beta=\pi$ the rotations are indistinguishable, which does not cause any problem since there is no ordering of summands in the definition of $\phi$.

We note that if $\alpha(x,gy)=\pi$ or $\alpha(x,gy)=0$ for some $g\in G$, then the plane $A$ coincides with the plane containing $Gx$, which also coincides with plane containing $Gy$. The path $\phi$ is still well defined. If $\alpha(x,gy)=\pi$, then $$\phi(t)=\frac{1}{2}R^{\alpha(x,gy)}(t)+\frac{1}{2}R^{\beta(x,gy)}(t),$$ where the rotations by $\pi$ are in opposite directions, and there is no ordering between them. If $\alpha(x,gy)=0$, then $R^{\alpha(x,gy)}(t)$ is the identity and the direction of rotation of $R^{\beta(x,gy)}(t)$
is irrelevant since it has coefficient zero.

It is easy to check that each summand in the definition of $\phi$ depends continuously on $x$ and $y$. This implies that the $\mathcal B_{2r}$-path $\phi$  depends continuously on $(x,y)\in S^m\times S^m$.

We define a $\mathcal B_{2r}$-path $\psi$ from $Gx$ to $Gy$ in $L^m_r$ as the image of $\phi$ under the quotient map $q:S^m\to L^m_r$,
$$
\psi_{Gx,Gy}(t)=\sum_{g\in G}\left(\frac{\beta(x,gy)}{\pi}qR^{\alpha(x,gy)}(t)+\frac{\alpha(x,gy)}{\pi}qR^{\beta(x,gy)}(t)\right)
$$
Since $G$ acts by isometries, we have $R^{\alpha(gx,gy)}(t)=gR^{\alpha(x,y)}(t)$ and
$R^{\beta(gx,gy)}(t)=gR^{\alpha(x,y)}(t)$. Therefore $\psi$ does not depend on the choice of the representative $x\in Gx$ in definition of $\phi$. It is easy to see that the map
$$\Psi:L^m_r\times L^m_r\to\mathcal B_{2r}(P(X))$$ defined by $\Psi(Gx,Gy)$ to $\psi_{Gx,Gy}$ is continuous,
and $\psi_{Gx,Gy}(0)=Gx$ and $ \psi_{Gx,Gy}(1)=Gy$.

Now suppose that $r=2k$.  For any vector $z\in Gx$, its negative $-z\in Gx$. We  consider lines $\ell_x$ and $\ell_{gy}$ through the vectors $x$ and $gy$ respectively, and define $\alpha(\ell_x,\ell_{gy})$ and $\beta(\ell_x,\ell_{gy})$ to be the angels between $\ell_x$ and $\ell_{gy}$. Thus,
$\alpha(\ell_x,\ell_{gy})+\beta(\ell_x,\ell_{gy})=\pi.$ Let $H\subset G$ be a subgroup of index 2. We define a $\mathcal B_{2k}$-path (in $\mathbb RP^{2n-1}$) from the line $\ell_x$ to $\sum_{g\in H}\frac{1}{k}\ell_{gy}$ as
$$
\phi(t)=\sum_{g\in H}\left(\frac{\beta(x,gy)}{\pi}R^{\alpha
(x,gy)}(t)+\frac{\alpha(x,gy)}{\pi}R^{\beta(x,gy)}(t)\right).
$$
This defines a $\mathcal B_{2k}$-path in $\mathbb S^m$ from $\frac{1}{2}x+\frac{1}{2}(-x)$ to $\sum_{g\in G}\frac{1}{2k}gy\in\mathcal B_r(Gy)$.
Then we define $\Psi$ as above by projecting to $L_r^m$.
\end{proof}

\subsection{Connectivity of $\mathcal B_n(X)$}
Kallel and Karouri proved that for any connected CW-complex $X$,  the space
$\mathcal{B}_n(X)$ is simply connected for all $n\ge 2$~\cite{KK}.  In~\cite{Dr2} it was shown that  the space
$\mathcal{B}_n(X)$ is  $(n-2)$-connected for any CW-complex $X$ and all $n$.
Also, Kallel and Karouri proved that for $k$-connected  CW-complex $X$, the space
$\mathcal{B}_n(X)$ is $(2n+k-2)$-connected for all $n$. There was a much earlier  result by Nakaoka~\cite{Na}, 
which states that $\mathcal{B}_n(X)$ is $(n+k-1)$-connected. We note that Nakaoka's result is sufficient for the purpose of our paper.

Since every ANR space is homotopy equivalent to a CW-complex, all of the above results hold for ANR spaces.

We recall that a map $f:X\to Y$ between ANR spaces is called an $n$-equivalence if it induces isomorphisms of homotopy groups in dimension $<n$ and an epimorphism in dimension $n$.
A typical example of an $n$-equivalence is a map with $(n-1)$-connected fibers. There is also a combinatorial version of this example.
\begin{prop}\label{V-B}
 A surjective map $f:X\to Y$ onto a finite dimensional simply connected simplicial complex $Y$,  such that $f^{-1}(\sigma)$ is an $(n-1)$-connected ANR for each simplex $\sigma\subset Y$, is an $n$-equivalence.
\end{prop}
The homological version of this statement is known as the Combinatorial Vietoris-Begle Theorem~\cite{Dr2}.
The Combinatorial Vietoris-Begle Theorem and the relative Hurewicz theorem implies that $f$ in Proposition~\ref{V-B} is an $n$-equivalence.

\

Let $q_m:\bigsqcup_{i=1}^mX_j\to\{1,\dots, m\}$ be the natural projection, $q_m(X_j)=j$. We note that for a finite set $C$ with $|C|=m$, we have $\mathcal B_n(C)=\Delta(C)^{(n-1)}$.
We denote by $$q_{n,m}=\mathcal B_n(q_m):\mathcal B_n(\bigsqcup_{j=1}^mX_j)\to\Delta(C)^{(n-1)}.$$
An important example of such situation is the projection onto path components $h:\Omega(X)\to\pi_1(X)$ of the loop space of $X$, when $X$ has a  finite fundamental group.

\begin{prop}\label{connectivity}
Suppose that all  ANR spaces $X_j$, $j=1,\dots m$, are $k$-connected for $k\ge 1$. Then the map $q_{n,m}$ is an $(n+k)$-equivalence for all $m$.

In particular, for $m\le n$ the space $\mathcal  B_n(X_1\sqcup\dots\sqcup X_m)$ is $(n+k-1)$-connected as it is the preimage of  a simplex; and   for $m>n$ it is $(n-2)$-connected as it is the preimage of an $(n-2)$-connected space $(\Delta^{m-1})^{(n-1)}$.
\end{prop}
\begin{proof}
We use induction on $n$. We will show that for every simplex $\sigma$  of dimension $\le n-1$, the preimage $q_{n,m}^{-1}(\sigma)$ is $(n+k-1)$-connected, and then apply Proposition~\ref{V-B}. Since for $m>2$ the space $\Delta(C)^{(1)}$ is not simply connected, we have to treat the case $n=2$ separately.  
The shortest way to do this is to note that the relative Hurewicz theorem, and hence Proposition~\ref{V-B}, applies to this map, since $\Delta(C)^{(1)}$ has trivial homotopy groups in dimension $>1$~\cite{Ha}.

For $n=1$ the statement of the proposition is obvious.

We assume that $q_{n,m}$ is an $(n+k)$-equivalence for all $m$. Hence
$\mathcal  B_n(X_1\sqcup\dots\sqcup X_m)$ is $(n+k-1)$-connected for $m\le n$ and $(n-2)$-connected for $m>n$.

By induction on $m$ we show that $\mathcal  B_{n+1}(X_1\sqcup\dots\sqcup X_m)$  is $(n+k)$-connected for $m\le n+1$.
For $m=1$ it follows from Nakaoka's result~\cite{Na}. 
Suppose that $\mathcal  B_{n+1}(X_1\sqcup\dots\sqcup X_m)$ are $(n+k)$-connected for all $s\le m\le n$.
Then $$\mathcal  B_{n+1}(X_1\sqcup\dots\sqcup X_m\sqcup X_{m+1})=\bigcup_{i=0}^{n+1}\mathcal  B_{i}(X_1\sqcup\dots\sqcup X_m)\ast \mathcal  B_{n-i+1}(X_{m+1}).$$
We note that the spaces
$\mathcal  B_{i}(X_1\sqcup\dots\sqcup X_m)$  are $(i-2)$-connected for $2\le i<n$.  The space $\mathcal  B_n(X_1\sqcup\dots\sqcup X_m)$ is $(n+k-1)$-connected
by the external induction assumption.
By  the Kallel-Karouri result,   $\mathcal  B_{n-i+1}(X_{m+1})$ is $(2(n-i+1)+k-2)$-connected. We recall that the join of an $\ell$-connected space with an $r$-connected space is $(\ell+r+2)$-connected.
Thus, we obtain that $$\mathcal  B_{i}(X_1\sqcup\dots\sqcup X_m)\ast \mathcal  B_{n-i+1}(X_{m+1})$$ is $(n+k)$-connected, since
$n+k\le (i-2)+(2(n-i+1)+k-2)+2\ \ \text{for}\ \ 0\le i\le n$.

We denote by  $M_i=\mathcal  B_{i}(X_1\sqcup\dots\sqcup X_m)\ast \mathcal  B_{n-i+1}(X_{m+1})$.  We have already checked that $M_i$ is $(n+k)$-connected for $1\le i\le  n$. Note that
$M_{n+1}$ is $(n+k)$-connected by the internal induction assumption.

The same argument shows that for $\ell\ge 1$ the space
$$(M_0\cup\dots\cup M_\ell)\cap M_{\ell+1}=\mathcal  B_{\ell-1}(X_1\sqcup\dots\sqcup X_m)\ast \mathcal  B_{n-\ell+1}(X_{m+1})$$ is $(n+k-1)$-connected.

Since all $M_\ell$ are  $(n+k)$-connected and the intersections  $(M_0\cup\dots\cup M_\ell)\cap M_{\ell+1}$ are $( n+k-1)$-connected, the union $\mathcal  B_{n+1}(X_1\sqcup\dots\sqcup X_m)=M_0\cup\dots\cup M_{n+1}$ is $(n+k)$-connected.

 We have proved that for every simplex $\sigma\subset\Delta(C)$ of dimension $\le n$, the preimage
$q_{n+1,m}^{-1}(\sigma)$ is $(n+k)$-connected. Then, by Proposition~\ref{V-B},  $q_{n+1,m}$ is an $(n+k+1)$-equivalence.
\end{proof}
 
Let $G$ be a group and $BG$  its classifying CW-complex. Then the loop space $\Omega(BG)$ is homotopy equivalent to a discrete space $G$.
Let $h:\Omega(BG)\to G$ denote this homotopy equivalence, and let $u:BG^{(m)}\to BG$ be the inclusion of the $m$-skeleton.

\begin{prop}\label{q_G}
Let $G$ be a finite group, and let
$q_G=h\Omega(u):\Omega(BG^{(m)})\to G$ be the composition of the maps defined above. 
Then the map $$q_{n,G}=\mathcal B_n(q_G):\mathcal B_n(\Omega(BG^{(m)}))\to\mathcal B_n(G)=\Delta(G)^{(n-1)}$$ is an $(n+m-2)$-equivalence.
\end{prop}
\begin{proof}
We note that $q_G$ is the projection of the union $\bigsqcup_{j\in G}X_j$ onto $G$, where each $X_j$ is homotopy equivalent to $\Omega(EG^{(m)})$, and
$EG$ is the universal cover of $BG$ with the induced CW-complex structure.
Since $EG^{(m)}$ is $(m-1)$-connected, each $X_j$ is an $(m-2)$-connected ANR. We apply
Proposition~\ref{connectivity} to the map $q_G$ to conclude that $q_{n,G}$
 is an $(n+m-2)$-equivalence.
\end{proof}

\begin{thm}\label{upper dcat for skeleta}
For a finite group $G$, $d\cat(BG^{(m)})\le |G|-1$ for all $m$.
\end{thm}
\begin{proof}
We recall our notations from the Preliminary section for fibrations $p_n:E_nX\to X$ and $p_n^G:E_nG\to BG$, which define $d\cat$ of spaces and groups, respectively.
Let $X=BG^{(m)}$. We consider the pull-back diagram with $n=|G|-1$
\[
\xymatrix{E_n(X)\ar[dr]_{p_n}\ar[r]^{q'} 
& u^*E_nG\ar[d]^{\xi}\ar[r]^{u'} & E_n(BG)\ar[d]_{p_n}\ar[r]^h_\sim & E_nG\ar[dl]^{p_n^{G}} \\
& X\ar[r]^u & BG &}
\]
where $\xi=u^*p_n^G$.
It suffices to show that $q'$ is an $m$-equivalence. The map $q'$, restricted to the fibers, equals $$q_{n+1}:\mathcal B_{n+1}(\Omega(X)\to\mathcal B_{n+1}(G).$$ By Proposition~\ref{q_G} the map $q_{n+1}$ is an $((n+1)+m-2)$ -equivalence. Hence, the fiber of $q'$ is $(n+m-1)$-connected. Therefore,
$q'$ is an $(n+m )$-equivalence, which is an $m$-equivalence since $n\ge 0$. Consequently,  a section of $\xi$, obtained from the section of $p_n^G$,
can be lifted to $E_n(X)$, since $\dim X=m$. Then by Theorem~\ref{dcat} $$d\cat(X)\le n=|G|-1.$$
\end{proof}
\begin{cor}\label{upper dcat}
For lens spaces, $d\cat(L^m_r)\le p-1$ for all $m$ and  $r$.
\end{cor}
Similarly we can get upper bounds for $\dTC$. We recall that the defining fibrations for $\dTC$ are denoted as $\bar p_n:F_n(X)\to X\times X$ for spaces and
$\bar p_n^G:F_nG\to BG\times BG$ for groups.
\begin{thm}\label{TC upper}
For odd $r$, 
$$
\dTC(L^m_r)\le\begin{cases}r-1\ & if\ m\le r-1,\\ 
m\ & if\ m\le 2r-1,\\
2r-1\ & if\ m\ge 2r-1.\\
\end{cases}
$$
For even $r=2n$, $\dTC(L^m_{2n})\le n-1$ for all $m$.
\end{thm}
\begin{proof}
Let $X=L^m_r$.
We consider the diagram similar to one from Theorem~\ref{upper dcat for skeleta}:
\[
\xymatrix{F_n(X)\ar[dr]_{\bar p_n}\ar[r]^{q'} 
& u^*F_n\mathbb Z_r\ar[d]^{\xi}\ar[r]^{u'} & F_n(L^\infty_r)\ar[d]_{\bar p_n}\ar[r]^h_\sim & F_n\mathbb Z_r\ar[dl]^{\bar p_n^{\mathbb Z_r}} \\
& X\times X\ar[r]^u& L_r^\infty\times L_r^\infty &}
\]
where $\xi=u^*\bar p_n^{\mathbb Z_r}$. Again, the  map $q'$, when restricted to the fibers, equals $$q_{n+1}:\mathcal B_{n+1}(\Omega(X))\to\mathcal B_{n+1}(\mathbb Z_r),$$ which, by Proposition~\ref{q_G},  is an $((n+1)+m-2)$-equivalence. Therefore, $q'$ is an $(n+m)$-equivalence. 

When $n\ge m$, the map $q'$ is a $2m$-equivalence. When $m\le r-1$, we have $n\ge m$ with $n=r-1$.
Since $\dTC(\mathbb Z_r)\le r-1$~\cite{KW1} and $\dim(X\times X)=2m$, we obtain that the map
$$
\bar p_n: F_n(X)\to X\times X$$ admits a section.

If $m\ge r-1$, then $n=m\ge r-1$. Therefore $p_n^{\mathbb Z_r}$ admits a section.
Since $n+m=2m$, the map $q'$ is a $(2m)$-equivalence. Hence, the pull-back section
of $\xi$ can be lifted to a section of $\bar p_n: F_n(X)\to X\times X$.

The case of even $r$ follows from Proposition~\ref{2p-1}
 \end{proof}
Using the same technique, we can determine upper bounds for $d\cat$ of the product of lens spaces.
This estimate together with a lower bound obtained in the next section, will be used in the construction of counterexamples to the product formulas presented at the end of the paper.

\begin{prop}\label{loops}
The loop space of the product of $k$ lens spaces is the disjoint union
$$\Omega((L^m_r)^k)\cong \bigsqcup_{j=1}^{r^k} X_j$$ where each $X_j$ is homotopy equivalent to $\Omega((S^m)^k)$.
\end{prop}
\begin{proof}
Consider the universal covering $\pi:(S^m)^k\to (L^m_p)^k$. 
Let $x_0\in(L^m_p)^k$, set $J=\pi^{-1}(x_0)$, and fix $j_o\in J$. For each $j\in J$, consider the set of all loops $$X_j\subset\Omega((L^m_p)^k)$$
that lift to a path from $j_0$ to $j$. Clearly, $X_j$ is a path component of $\Omega((L^m_p)^k)$. Moreover, $X_j$ is homeomorphic to the space of all paths from $j_0$ to $j$,
and this space is homotopy equivalent to the loop space $\Omega((S^m)^k)$.
\end{proof}

\begin{prop}
For $k$, $m$, and $p$ satisfying $m\le \frac{p^k-1}{k-1}$, $$d\cat(L^m_p)^k\le p^k-1.$$
\end{prop}
\begin{proof}
Let $G=\Z_r^k$ and $X=L^m_r$. In the diagram
\[
\xymatrix{E_n(X^k)\ar[dr]_{p_n}\ar[r]^{q'} 
& u^*E_nG\ar[d]^{\xi}\ar[r]^{u'} & E_n(BG)\ar[d]_{p_n}\ar[r]^h_\sim & E_nG\ar[dl]^{p_n^G} \\
& X^k\ar[r]^u &( L_r^\infty)^k &}
\]
the map $q'$, restricted to the fiber equals $$q_{n+1}:\mathcal B_{n+1}(\Omega(X^k))\to\mathcal B_{n+1}(G)$$ which, by Proposition~\ref{q_G} and Proposition~\ref{loops},  is an $((n+1)+m-2)$ -equivalence. As above it follows that $q'$ is an $(n+m)$-equivalence. Hence, for $n\ge (k-1)m$, it is a $(km)$-equivalence.
A section of $\xi$, obtained from the section of $p^G_n$,
can therefore be lifted to $E_n(X^k)$, since $\dim X^k=km$. Thus,  $d\cat X^k\le n=r^k-1$ whenever $r^{k-1}\ge (k-1)m$ whenever $p^{k-1}\ge(k-1)m$.
\end{proof}

\section{Lower bounds for $d\cat$ and $\dTC$ of lens spaces}

\subsection{Borsuk--Ulam type theorems}
Suppose that a finite group $G$ acts freely on $X$ and let $f:X\to\mathbb R$ be a function. For $x\in X$, denote by $Gx$ its orbit.
By definition, the coincidence set of $f$ is
$$
A_f=\{x\in X\mid\ \ |f(Gx)|=1\}.
$$
Let $p:X\to X/G$ be the projection onto the orbit space and let $u:X/G\to BG$ be its classifying map.
\begin{prop}\label{BU}
Let $G$ be a finite group  acting freely on a finite CW complex $X$ and $|G|=n+1$. The following are equivalent:

{\rm (1)} $d\cat(u)\ge n$. 

{\rm (2)} $A_f\ne\emptyset$ for any continuous function $f:X\to\mathbb R$.

\end{prop}
\begin{proof}

(1) $\Rightarrow$ (2). Suppose that $A_f=\emptyset$ for some $f:X\to\mathbb R$. In view of Corollary~\ref{Lift} to get a contradiction with  $d\cat(u)\ge n$ we 
construct a $\mathcal B_n$-section of $p:X\to B/G$.

We note that the multi-valued map $p^{-1}:X/G\to X$ is continuous for the Hausdorff metric on $X$. Hence the multi-valued map $F:X/G\to\mathbb R$, defined by
$F(z)=f(p^{-1}(z))$, is continuous with respect to the Hausdorff metric on $\mathbb R$. Therefore, the functions $g_+,g_-:X/G\to\mathbb R$, 
defined by $g_+(z)=\max F(z)$ and $g_-(z)=\min F(z)$, are continuous.

Since $A_f=\emptyset$,  we have $g_-(z)<g_+(z)$ for all $z\in X/G$. Therefore, the graphs $G_+$ and $G_-$ of
$g_+$ and $g_-$ are closed disjoint subsets of $(X/G)\times\mathbb R$. Let $W_+$ and $W_-$ be disjoint open neighborhoods of $G_+$ and $G_-$.  We consider the map $(p,f):X\to (X/G)\times\mathbb R$ and define
$S^+=(p,f)^{-1}(W_+)$ and $Y=X\setminus S^+$. Let $S^+(z)=S^+\cap p^{-1}(z)$.
For $x\in X$, set $\lambda_x=d(x,Y)$, where  $d$ is a metric on $X$ and define $\Lambda_z=\sum_{x\in S^+(z)}\lambda_x$.

We define a probability measure section $\mu$ of $p$
by
$$
\mu(z)=\sum_{x\in S^+(z)}\frac{\lambda_x}{\Lambda_z}\delta_{x}=\sum_{x\in p^{-1}(z)}\frac{\lambda_x}{\Lambda_z}\delta_{x}$$
where the first equality is the definition and the second follows from the fact that $\lambda_x=0$ for $x\notin S^+$.

Since $|S^+(z)|\le n$, we obtain that $\mu(z)\in \mathcal B_n(p^{-1}(z))$. Since $\lambda_x$ depends continuously on $x$, the function $\mu$ is continuous.
Indeed, to show that $\mu$ is continuous at $z\in X/G$ we consider a neighborhood $V$ such that there is a homeomorphism $h:p^{-1}(V)\to V\times p^{-1}(z)$ compatible with $p$. 
The projection onto the second factor $p^{-1}(z)=\{x_0,\dots, x_n\}$ induces an ordering on the  summands in the sum defining $\mu$. Then $\mu$ is continuous as a finite sum of continuous functions.

(2) $\Rightarrow$ (1). Suppose that $d\cat(u)< n$.  Then by Corollary~\ref{Lift}
there is a $\mathcal B_n$-section $\mu:X/G\to X$ of $p$. We view the measure $\mu(z)$ on $p^{-1}(z)$ as a function $\mu(y):p^{-1}(z)\to\mathbb R$. Define
$$A^+(z)=\{x\in p^{-1}(z))\mid \mu(z)(x)=\max_{x'\in p^{-1}(z)}\mu(z)(x')\}.$$
Since $|p^{-1}(z)|=n+1$ and $|supp(\mu(z)|\le n$, the set $A^+(z)$ is a proper subset of $p^{-1}(z)$ for every $z\in X/G$.

We note if $\lim x^i_k=x^i$, $i=0,\dots, n$ and 
$$\lim_{k\to\infty}\mu_k=\sum_{i=0}^n t_k^i\delta_{x^i_k}=\mu=\sum_{i=0}^n t^i\delta_{x^i} $$
 in $\mathcal B_{n+1}(X)$, then $\lim t^i_k=t^i$ for all $i$. In particular, if for $t^i_k$ is the maximal value of $\mu_k$ for all $k$, then $t^i$ is the maximal value of $\mu$.
Hence if $z_k\to z$, then $\mu(z_k)\to\mu(z)$, and locally at $z$ all these measures can be written as ordered sums as above. Therefore, $\lim A^+(z_k)\subset A^+(z)$.
Hence, $$A=\bigcup_{y\in X/G}A^+(y)$$ is a closed subset of $X$. Let $f:X\to\mathbb R$ be defined as $f(x)=d(x,A)$, then $f^{-1}(0)=A$. Thus,  we obtain that $A_f=\emptyset$, which contradicts the assumption.
\end{proof}

We will use the following Borsuk--Ulam type results collected here in a single theorem.
\begin{thm}\label{BU type}
Suppose that a group $G$ acts freely on a  space $X$. Then $A_f\ne \emptyset$ if one of the following holds:

{\rm (1)} $G=\mathbb Z_p$ for prime $p$, and $X=S^m$ with $m\ge p-1$. 

\

{\rm(2)} $G=\mathbb Z_{p^k}$ for an odd prime $p$, and $X=S^m$ with $$m\ge kp^k-(k+2)p^{k-1}+3$$ or with $m\ge (k-1)2^{k-1}+1$ for $p=2$. 

\

{\rm(3)}  $G=(\mathbb Z_p)^k$, and $X$ is $(m-1)$-connected CW-complex for $m\ge p^k$. 
\end{thm}
Statement  (1) was proved in~\cite{Mu1}, (2) in~\cite{Mu2} for odd $p$ and in~\cite{Tu} for $p=2$, and (3) was proved in~\cite{Vo}.

\begin{prop}
For prime $p$ and $m\ge p-1$, $d\cat(L^m_p)\ge p-1$.
\end{prop}
\begin{proof}
Theorem~\ref{BU type}(1) and Theorem~\ref{BU}, applied to the $\mathbb Z_p$-action on $S^m$ that defines $L^m_p$, imply that
$d\cat(u)\ge p-1$ where $u:L^m_p\to L^\infty_p$ is the inclusion. Then by inequality~\ref{dcat of maps}, $d\cat L^m_p\ge p-1$.
\end{proof}
As a corollary, in view of Corollary~\ref{upper dcat}, we obtain the following
\begin{thm}\label{=}
For prime $p$ and $m\ge p-1$, $d\cat(L^m_p)= p-1$.
\end{thm}
\begin{cor}
For prime $p$ and $m\ge p-1$, $\dTC(L^m_p)\ge p-1$.
\end{cor}
\begin{proof}
For $m\ge p$, $\dTC(L^m_p)\ge d\cat(L^m_p)\ge p-1$.
\end{proof}

\begin{prop}\label{k}
For prime $p$ and $m\ge p^k$, $d\cat(L^m_p)^k\ge p^k-1$.
\end{prop}
\begin{proof}
Consider the product action of $G=(\mathbb Z_p)^k$ on $(S^m)^k$ for odd $m$.
By Theorem~\ref{BU type}(3)  for $m\ge p^k$ we have $A(f)\ne\emptyset$ for any map $f:(S^m)^k\to\mathbb R$.
Then by Proposition~\ref{BU} $$d\cat(u:(L^m_p)^k\to (L^\infty_p)^k)\ge p^k-1.$$ Hence, $d\cat(L^m_p)^k\ge p^k-1$.
\end{proof}

\begin{thm}\label{p^k}
For prime $p$,
$d\cat(L^m_{p^k})= p^k-1$ when $p$ is odd and $m\ge kp^k-(k+2)p^{k-1}+3$, and
when $p=2$ and $m\ge (k-1)2^{k-1}+1$.
\end{thm}
\begin{proof}
Under these conditions on $m$ Theorem~\ref{BU type}(2) implies that $A_f\ne\emptyset$. Then by Proposition~\ref{BU}, we obtain that
$$d\cat(u:L^m_{p^k}\to L^\infty_{p^k})\ge p^k-1.$$ Therefore, $d\cat(L^m_{p^k})\ge p^k-1$.
Since by Corollary~\ref{upper dcat} $d\cat(L^m_{p^k})\le p^k-1$, we obtain that
$$d\cat(L^m_{p^k})= p^k-1.$$
\end{proof}

\subsection{Counterexamples to the product formula}
Counterexamples to the product formula for groups can be easily derived for both $d\cat$ and $\dTC$  from~\cite{DJ},\cite{KW1} and ~\cite{KW2}.
Namely,
$$d\cat(\mathbb Z_2\times\mathbb Z_2)=3> 1+1=d\cat\mathbb Z_2+d\cat\mathbb Z_2\ \ \text{and}$$
$$\dTC(\mathbb Z_2\times\mathbb Z_2)\ge d\cat(\mathbb Z_2\times\mathbb Z_2)=3> 1+1=\dTC(\mathbb Z_2)+\dTC(\mathbb Z_2).$$
Here we give counterexamples for manifolds.
\begin{prop}
For a prime $p$ and $m\ge p^2$
$$
d\cat(L_p^m\times L^m_p)> d\cat(L_p^m)+d\cat(L_p^m).
$$
If, in addition, $p\ge 5$  then
$$
\dTC(L_p^m\times L_p^m) > \dTC(L^m_p)+\dTC(L^m_p).
$$
\end{prop}
\begin{proof}
By Proposition~\ref{k}, for $m\ge p^2$, 
$
d\cat(L_p^m\times L^m_p)\ge p^2-1$.
By Theorem~\ref{=},  for $m\ge p-1$, we have $d\cat(L^m_p)=p-1$.
Since $p^2-1>2(p-1)$, the result follows.

Since $\dTC\ge d\cat$, by Proposition~\ref{k}, for $m\ge p^2$, 
$
\dTC(L_p^m\times L_p^m)\ge p^2-1.$
By Theorem~\ref{TC upper}, for $m\ge 2p-1$, we have $\dTC(L^m_p)\le 2p-1$. Since for $p\ge 5$, $p^2-1>4p-2$, we obain 
$$\dTC(L_p^m\times L_p^m) > \dTC(L^m_p)+\dTC(L^m_p).$$
\end{proof}

\section{Acknowledgments} The author  thanks
FIM at ETH, Z\"urich and MPI for Mathematics in Bonn for hospitality.


\begin{thebibliography}{[FGLO]}
\bibliographystyle{amsalpha}



\bibitem{ABS} M. Aguiar, N. Bergeron, and F. Sottile. Combinatorial Hopf algebras and generalized Dehn-Sommerville
equations. Compos. Math., 142 (2006), 1-30.


\bibitem{Ba} G. Baumslag. Some aspects of groups with unique roots. Acta Math., 10 (1960), 277-303.

\bibitem{BH} M. Bridson, A. Hawfliger, Metric spaces of non-positive curvature, Springer, 1999.
 
\bibitem{Bro}
K. Brown, Cohomology of Groups. \emph{Graduate Texts in Mathematics},
\textbf{87} Springer, New York Heidelberg Berlin, 1994.

\bibitem{C}
  G. Carlsson. Equivariant stable homotopy and Sullivan's conjecture, 
  Invent. Math. 103 (1991), no. 3, 497–525.

\bibitem{CV}  D. C. Cohen, L. Vandembroucq, Topological complexity of the Klein bottle, J Appl. and Comput. Topology 1 (2017), 199-213.
  
\bibitem{CLOT} O. Cornea, G. Lupton, J. Oprea, D. Tanré, \emph{Lusternik-Schnirelmann category}, Math. Surveys Monogr. 103, AMS,  2003. 



\bibitem{DaJ} N. Daundkar, E. Jauhari, {\em On the complexity of parametrized motion planning algorithms} Preprint, 2025 arXiv:2508.17629

\bibitem{Dr1}  A. Dranishnikov, {\em On dimension of product of groups}.  Algebra and Discrete Mathematics,  28  (2019), no 2,  47-56.


\bibitem{Dr2} A. Dranishnikov,  Distributional topological complexity of finitely generated groups, Contemp. Math., 830 American Mathematical Society, 2025, 83-105.

 
\bibitem{DJ} A.~Dranishnikov, E.~Jauhari, Distributional topological complexity and LS-category. In \emph{Topology and AI}, ed. M.~Farber \emph{et al.}, EMS Ser. Ind. Appl. Math., 4, EMS Press, Berlin, 2024, pp. 363–385. 



\bibitem{DMN}
 Dwyer, William; Haynes Miller; Joseph Neisendorfer. 
  Fibrewise Completion and Unstable Adams Spectral Sequences,
 Israel J. Math. 66 (1989), no. 1-3, 160–178.

\bibitem{EG} S. Eilenberg, T. Ganea, {\em On the Lusternik-Schnirelmann Category of Abstract Groups.} Annals of
Mathematics, 65, (1957), 517-518.


\bibitem{F1} M.~Farber, Topological complexity of motion planning. Discrete Comput. Geom. 29 (2003), no.~2, 211-221. 


\bibitem{F2} M.~Farber, Invitation to topological robotics. Zurich Lectures in Advanced Mathematics. European Mathematical Society (EMS), Zurich, 2008. 

\bibitem{F3} M.~Farber, Topology of robot motion planning, NATO Sci. Ser. II Math. Phys. Chem., 217
Springer, Dordrecht, 2006, 185-230.


\bibitem{FM} M.~Farber, S~Mescher, {\em On the topological complexity of aspherical spaces}. Journal of Topogy and Analysis, 2019, arXiv:1708.06732v2 [math.AT].

\bibitem{FGLO} M. Farber, M. Grant, G. Lupton, and J. Oprea, {\em Bredon cohomology and robot motion planning.} Algebr. Geom. Topol. 19 (2019), no. 4, 2023–2059.

 \bibitem{FTY} M.~Farber, S.~Tabachnikov, S.~Yuzvinsky, Topological robotics: motion planning in projective spaces. \emph{Intl. Math. Res. Not.} (2003), no.~34, 1853--1870.

\bibitem{Gr} M. Gromov, Hyperbolic groups, Essays in group theory, 75-263, Math. Sci. Res. Inst. Publ., 8, Springer, New York, 1987.

\bibitem{Ha} A. Hatcher, Algebraic Topology, Cambridge, 2002.

\bibitem{HLR} Ch. J. Hillar, L. Levine,  D. Rhea, Equations Solvable by  Radicals in a
 Uniquely Divisible Group, Bull. Lond. Math. Soc. 45 (2013), no 1, 61-79.

\bibitem{J1} E. Jauhari, Distributional category of manifolds, Bol. Soc. Mat. Mex. (3) 31 (2025), 63, pp. 34.

\bibitem{J2} E. Jauhari, On sequential version of distributional topological complexity, Topology Appl. 363 (2025), 109271, pp. 28.

\bibitem{JO1} E. Jauhari, J. Oprea, Bochner-type theorems for distributional category, Preprint 2025, arXiv:2505.21763.

\bibitem{JO2} E. Jauhari, J. Oprea, On distributional one-category, diagonal distributional complexity, and related invariants, Preprint 2025, arXiv:2508.06973.

\bibitem{KK} S. Kallel, R.Karouri,  {\em Symmetric joins and weighted barycenters}  Adv. Nonlinear Stud. (1), \textbf{11} (2011), 117-143.

\bibitem{KW1} B. Knudsen, S. Weinberger, {\em Analog category and complexity}, SIAM J. Appl. Algebra Geom. 8 (2024), no. 3, 713–732.

\bibitem{KW2} B. Knudsen, S. Weinberger, {\em On the analog category of finite groups}, Preprint 2024 to appear in AGT.

\bibitem{La}
  Lannes, Jean. Sur les espaces fonctionnels dont la source est le classifiant d'un p-groupe abélien élémentaire,
Inst. Hautes Études Sci. Publ. Math. No. 75 (1992), 135–244.

\bibitem{Lu}  W. Luck, Survey on classifying spaces for families of subgroups, Infinite groups: geometric, combinatorial
and dynamical aspects, 269-322, Progr. Math., 248, Birkhuser, Basel, 2005.

\bibitem{LS}  L.~Lusternik, L.~Schnirelmann, Sur le probl\`eme de trois g\'eodesiques ferm\'ees sur les surfaces de genre 0 (in French). \emph{C. R. Acad. Sci.  Paris} \textbf{189} (1929), 269--271.


\bibitem{Mu1} H.J. Munkholm, Borsuk-Ulam type theorems for proper $\mathbb Z_p$-actions on (mod $p$ homology) $n$-spheres, Math Scand 24 (1969) 167-185.

\bibitem{Mu2} H.J. Munkholm, Borsuk-Ulam theorem for proper $\mathbb Z_{p^k}$-action on $S^{2n-1}$ and maps $S^{2n-1}\to\mathbb R^m$, Osaka J. Math. 7
(1970) 451-456.

\bibitem{Na} M. Nakaoka, Cohomology of symmetric products, J. Institute of Polytechnics, Osaka city University 8,
no. 2, (1956) 121-140.


 \bibitem{Tu} Yu. Turygin,  Borsuk-Ulam property of finite groups actions 
  on manifolds and applications.  Ph.D. thesis, University of Florida, 2007.

\bibitem{V} V. Vassiliev, Topological order complexes and resolutions of discrimant sets, Publications Institut
Math´ematiques {\bf 66}  80 (1999), 165-185.

\bibitem{Vo} A. Volovikov, On index of $G$-spaces, Sb. Math. 191 (2000) 1259-1277.

\end{thebibliography}
\end{document}